\documentclass{amsart}[14pt,a4paper]
\usepackage{geometry}
\usepackage{amsfonts}
\usepackage{hyperref}
\usepackage{amsmath}
\usepackage{galois}
\DeclareMathSizes{10}{10}{7}{5}
\usepackage{mathtools}
\numberwithin{equation}{section}

\usepackage{enumitem}
\usepackage{amssymb}

\numberwithin{equation}{section}

\usepackage{enumitem}
\usepackage{amssymb}
\usepackage{mathrsfs}

\newtheorem{proposition}{Proposition}[section]
\newtheorem{theorem}{Theorem}[section]

\newtheorem{lemma}{Lemma}[section]
\newtheorem{corollary}{Corollary}[section]

\begin{document}
\title{On the area rate of perturbed Siegel disks}
\author{Jianyong Qiao and Hongyu Qu\textsuperscript{*}}
\address[Jianyong Qiao]{School of Science, Beijing University of Posts and Telecommunications, Beijing
100786, P. R. China}
\address[Hongyu Qu]{School of Science, Beijing University of Posts and Telecommunications, Beijing
100786, P. R. China}
\address[Hongyu Qu]{Yau Mathematical Sciences Center, Tsinghua University,
Beijing 100084, P. R. China}
%\address{Yau Mathematical Sciences Center, Tsinghua University, Beijing
%100084, P. R. China. \textit{Email:} \textit{hongyuqu@mail.tsinghua.edu.cn}}
\renewcommand{\thefootnote}{\fnsymbol{footnote}}
\footnotetext[1]{Corresponding author, Email: hongyuqu2022@126.com}
\footnotetext[2]{The research work was supported by
the National Natural Science Foundation of China under Grants No.12071047 and
the National Key R\&D Program of China under Grant 2019YFB1406500.}

\maketitle

\begin{abstract}
In the famous work by Buff and Ch\'eritat constructing quadratic Julia sets with positive area,
the control of the shape of perturbed Siegel disks is a key technique.
To do it, Buff and Ch\'eritat added a restrictive condition on rotation numbers.
In this paper, we show that this restrictive condition can be deleted.
As a consequence, more quadratic Julia sets with positive area can be constructed.
\end{abstract}

\section{Introduction}
Let the quadratic polynomial
$$P_{\alpha}(z)=e^{2\pi i\alpha}z+z^2,$$
where $0<\alpha<1$ is an irrational number with continued fraction expansion
$$\alpha=[0;a_1,a_2,\cdots]=\frac{1}{a_1+\frac{1}{a_2+\frac{1}{\ddots}}}.$$
According to \cite{Br} or \cite{Yoc95}, when $\alpha$ is a Brjuno number,
$P_{\alpha}$ has a Siegel disk centering at $0$, denoted by $\Delta_{\alpha}$.
%We denote by $\Delta_{\alpha}$ this Siegel disk.
Assume
\[\alpha_n:=[0;a_1,\cdots,a_n,A_n,t_1,t_2,\cdots],\]
here $\{A_n\}_{n=1}^{+\infty}$ is a sequence of positive integers and
$\theta:=[0;t_1,t_2,\cdots]$ is a Brjuno number.
Then $\alpha_n$ is also a Brjuno number.
%Thus $P_{\alpha_n}$ has a Siegel disk $\Delta_{\alpha_n}$ centering at $0$.
Let $\Delta'_n$
be the Siegel disk of the restriction of $P_{\alpha_n}$ to $\Delta_{\alpha}$ centering at $0$
and define
$${\rm dens}_U(X):=\frac{{\rm area}(U\cap X)}{{\rm area}(U)}$$
for any two measurable sets $U,X\subset\mathbb{C}$ with $0<{\rm area}(U)<+\infty$.
In \cite{BC12}, Buff and Ch\'eritat proved the following result:
\begin{theorem}
\label{T1.2}Let $p_n/q_n$ be the approximants to $\alpha$. If
the sequence $\{A_n\}_{n=1}^{+\infty}$ satisfies the following condition:
\begin{equation}
\label{F2}\limsup_{n\to+\infty}\sqrt[q_n]{\log{A_n}}\leq1,
\end{equation}
then for any nonempty open set $U\subset\Delta_{\alpha}$,
\begin{equation*}
\label{F1}\liminf_{n\to+\infty}{\rm dens}_{U}(\Delta'_n)\geq\frac{1}{2}.
\end{equation*}
\end{theorem}
In \cite{BC12}, Theorem \ref{T1.2} is one of three key techniques to construct
quadratic Julia sets with positive area.
In that paper, Buff and Ch\'eritat told without proof that they think the condition (\ref{F2}) is not needed.
The present paper is denoted to proving this conjecture.
We will prove
\begin{theorem}
\label{T1.1}Assume $\alpha:=[0;a_1,\cdots]$ and $\theta:=[0;t_1,\cdots]$
are Brjuno numbers and
\[\alpha_n:=[0;a_1,\cdots,a_n,A_n,t_1,t_2,\cdots]\]
with $\{A_n\}_{n=1}^{+\infty}$ any sequence of positive integers.
%Let $\Delta_{\alpha}$ be the Siegel disk of $P_{\alpha}$ and let $\Delta'_n$
%be the Siegel disk of the restriction of $P_{\alpha_n}$ to $\Delta_{\alpha}$.
Then we have that for any nonempty open set $U\subset\Delta_{\alpha}$,
\begin{equation*}
\liminf_{n\to+\infty}{\rm dens}_{U}(\Delta'_n)\geq\frac{1}{2}.
\end{equation*}
\end{theorem}
Theorem \ref{T1.1} can help us get more quadratic Julia sets with positive area than the original result by Buff and Ch\'eritat,
for example, let $N$ be a positive integer and $\{n_j\}_{j=1}^{+\infty}$
be an increasing sequence of positive integers. We define another sequence of positive integers $\{b_{n_j}\}_{j=1}^{+\infty}$ as follows:
for any positive integer $j$,
define
\begin{flalign*}
b_{n_j}:=&\ [0;a_{j1},a_{j2},\cdots,a_{jn_j}]\\
=&\ \frac{1}{a_{j1}+\frac{1}{
\begin{matrix}a_{j2}+&\quad&\quad\\
\quad&\ddots\quad\\
\quad&\quad&+\frac{1}{a_{jn_j}}\end{matrix}}}\\
=&\frac{p_{n_j}}{q_{n_j}}
\end{flalign*}
with positive integers $p_{n_j},q_{n_j}$ which are relatively prime and
$$a_{jm}=\left\{\begin{matrix}3^{2^{q_{n_t}}},&m=n_t+1,1\leq t\leq j-1\\
N,&else\end{matrix}\right..$$
Let $\alpha_0$ be the limit of $\{b_{n_j}\}_{j=1}^{\infty}$, that is
$$\alpha_0=[0;a_1,a_2,\cdots,a_m,\cdots],\ a_m=\left\{\begin{matrix}3^{2^{q_{n_j}}},&m=n_j+1,j\geq1\\
N,&else\end{matrix}\right..$$
Then based on Theorem \ref{T1.1} and the method used by Buff and Ch\'eritat \cite{BC12},
one can prove that there exist $N$ and $\{n_j\}_{j=1}^{+\infty}$ such that
$P_{\alpha_0}$ has Julia set with positive area.
%The main idea of the proof of the first half part of Theorem \ref{T1.1} is same as that of Buff and Ch\'eritat,
%except using near-parabolic renormalization instead of Yoccoz's renormalization.

\section{Applying parabolic explosion to perturbed Siegel disks}
Parabolic explosion technique is established by Ch\'eritat and is applied by Buff and Ch\'eritat
to problems about perturbed Siegel disks (see \cite{ABC}\cite{BC04}\cite{BC06}\cite{BC12}\cite{Che00} for details).
In this section, parabolic explosion technique is applied to reduce
the proof of Theorem \ref{T1.1} to the proof of the following proposition \ref{P1}.
This idea is due to Buff and Ch\'eritat (see \cite{BC12}).
%All results of this section are due to Buff and Ch\'eritat (see []).

Let $\alpha$ be a Brjuno number with its approximants $p_k/q_k$ and
$\mathbb{D}_r:=\{z\in\mathbb{C}:|z|<r\}$ for all $r>0$.
Then $P_{\alpha}$ has a Siegel disk $\Delta_{\alpha}$ centering at $0$ and
%In general, $\Delta_{\alpha}$ depends very sensitively on rotation number $\alpha$.
there exist a unique positive real number $r_{\alpha}$ and a unique conformal isomorphism
$\phi_{\alpha}:\mathbb{D}_{r_{\alpha}}\to\Delta_{\alpha}$ such that
$\phi_{\alpha}(0)=0$, $\phi_{\alpha}'(0)=1$ and
$P_{\alpha}\comp\phi_{\alpha}(z)=\phi_{\alpha}(e^{2\pi i\alpha}z)$ for all $z\in\mathbb{D}_{r_{\alpha}}$.
According to \cite{BC12}, there exist a sequence of positive real numbers $\{\rho_k\}_{k=1}^{\infty}$ and
a sequence of holomorphic functions (fixing $0$) $\chi_k:\mathbb{D}_{\rho_k}\to\mathbb{C}$ such that
\begin{itemize}
\item[$\bullet$] $\lim_{k\to\infty}\rho_k=1$;
\item[$\bullet$] If $\delta\in\mathbb{D}_{\rho_k}\setminus\{0\}$ and
if we set $\zeta:=e^{2\pi ip_k/q_k}$ and $\eta:=\frac{p_k}{q_k}+\delta^{q_k}$,
then $\chi_k(\delta),\chi_k(\zeta\delta),\cdots,\chi_k(\zeta^{q_k-1}\delta)$ form a cycle of period $q_k$ of $P_{\eta}$.
In particular, for all $\delta\in\mathbb{D}_{\rho_k}$, $\chi_k(\zeta\delta)=P_{\eta}(\chi_k(\delta))$;
\item[$\bullet$] $|\chi'_k(0)|\to r_{\alpha}$ as $k\to\infty$;
\item[$\bullet$] the sequence of maps $\psi_k:\delta\to\chi_k(\delta/\chi'_k(0))$ converges uniformly
on every compact subset of $\mathbb{D}_{r_{\alpha}}$ to $\phi_{\alpha}$.
\end{itemize}
Let the continued fraction expansion of $\alpha$ be
$\alpha=[0;a_1,\cdots]$,
$\theta:=[0;t_1,\cdots]$ be
a Brjuno number and
\[\alpha_n:=[0;a_1,\cdots,a_n,A_n,t_1,t_2,\cdots]\]
with $\{A_n\}_{n=1}^{+\infty}$ any sequence of positive integers.
For all $0<r<r'<1$, it follows from the above properties that
for large enough $n$,
\begin{itemize}
\item[$\bullet$] $\chi_n$ is well defined and univalent on $\mathbb{D}_{r'}$;
\item[$\bullet$] $f_n(z):=\chi_n^{-1}\comp P_{\alpha_n}\comp\chi_n(z)$ is well defined on $\mathbb{D}_r$ and
$f_n(z)$ converges uniformly to $e^{2\pi i\alpha}z$ on $\mathbb{D}_r$.
\end{itemize}
Let $\Delta_n(r)$
be the Siegel disk of the restriction of $f_n$ to $\mathbb{D}_r$ centering at $0$,
$H(r_1,r_2):=\{z\in\mathbb{C}:r_1<|z|<r_2\}$ for all $0\leq r_1<r_2$.
Then the following proposition implies Theorem \ref{T1.1}.
\begin{proposition}
\label{P1}
For any subsequence $\{n_k\}_{k=1}^{\infty}$ with
$$\sqrt[q_{n_k}]{A_{n_k}}\xrightarrow{k\to+\infty}A\in(1,+\infty]$$
and for any nonempty open set $U\subset H(1/A,r'')$ with any $1/A<r''<r<1$, we have
\begin{equation}
\label{F2a}\liminf_{k\to+\infty}{\rm dens}_U\Delta_{n_k}(r)\geq\frac{1}{2}.
\end{equation}
\end{proposition}
The proof of Proposition \ref{P1} is postponed to the next section, and here
we only want to show how Proposition \ref{P1} implies Theorem \ref{T1.1}.

We suppose that Proposition \ref{P1} holds.
Given any subsequence $\{n_k\}_{k=1}^{\infty}$ with
\begin{equation}
\label{F2.1}\sqrt[q_{n_k}]{1/A_{n_k}}\xrightarrow{k\to+\infty}1/A\leq1,
\end{equation}
it follows from [Proposition 2, the remark following Proposition 2, Corollary 3, \cite{ABC}] that
for any nonempty open set $U\subset\phi_{\alpha}(\mathbb{D}_{r_{\alpha}/A})$,
\begin{equation}
\label{F2b}\lim_{k\to+\infty}{\rm dens}_U\Delta'_{n_k}=1.
\end{equation}
If $1/A=1$, then for any nonempty open set $U\subset\Delta_{\alpha}$,
(\ref{F2b}) also holds.
%\[\lim_{k\to+\infty}{\rm dens}_U\Delta'_{n_k}=1.\]
If $1/A<1$, then
for any $1/A<r''<r<1$,
it follows from Proposition \ref{P1} that
for any nonempty open set $U\subset H(1/A,r'')$, (\ref{F2a}) holds.
%$$\liminf_{k\to+\infty}{\rm dens}_U\Delta_{n_k}(r)\geq\frac{1}{2}.$$
Thus for any nonempty open set $U\subset\phi_{\alpha}(H(r_{\alpha}/A,r_{\alpha}r''))$,
$$\liminf_{k\to+\infty}{\rm dens}_U\phi_{\alpha}(r_{\alpha}\Delta_{n_k}(r))\geq\frac{1}{2}.$$
Since $\chi_{n_k}(z)$ converges uniformly to $\phi_{\alpha}(r_{\alpha}z)$ on every compact subset of $\mathbb{D}_1$
as $k\to\infty$, we have that
for any nonempty open set $U\subset\phi_{\alpha}(H(r_{\alpha}/A,r_{\alpha}r''))$,
$$\liminf_{k\to+\infty}{\rm dens}_U\chi_{n_k}(\Delta_{n_k}(r))=
\liminf_{k\to+\infty}{\rm dens}_U\phi_{\alpha}(r_{\alpha}\Delta_{n_k}(r))\geq\frac{1}{2}.$$
Then for any nonempty open set $U\subset\phi_{\alpha}(H(r_{\alpha}/A,r_{\alpha}r''))$,
$$\liminf_{k\to+\infty}{\rm dens}_U\Delta'_{n_k}\geq
\liminf_{k\to+\infty}{\rm dens}_U\chi_{n_k}(\Delta_{n_k}(r))\geq\frac{1}{2}.$$
So
for any nonempty open set $U\subset\phi_{\alpha}(H(r_{\alpha}/A,r_{\alpha}r''))$,
\begin{equation}
\label{F2c}\liminf_{k\to+\infty}{\rm dens}_U\Delta'_{n_k}\geq\frac{1}{2}.
\end{equation}
Letting $r''\to1^-$, we obtain that for any nonempty open set $U\subset\phi_{\alpha}(H(r_{\alpha}/A,r_{\alpha})),$
(\ref{F2c}) also holds.
%\[\liminf_{k\to+\infty}{\rm dens}_U\Delta'_{n_k}\geq\frac{1}{2}.\]
For any nonempty open set $U\subset\Delta_{\alpha}$,
if
$${\rm area}(U\cap\phi_{\alpha}(\mathbb{D}_{r_{\alpha}/A}))=0\ {\rm or}\
{\rm area}(U\cap\phi_{\alpha}(H(r_{\alpha}/A,r_{\alpha})))=0,$$
(\ref{F2c}) holds;
otherwise,
\begin{flalign*}
\liminf_{k\to+\infty}{\rm dens}_U\Delta'_{n_k}\geq&
\liminf_{k\to+\infty}\min\{{\rm dens}_{U\cap\phi_{\alpha}(\mathbb{D}_{r_{\alpha}/A})}\Delta'_{n_k},
{\rm dens}_{U\cap\phi_{\alpha}(H(r_{\alpha}/A,r_{\alpha}))}\Delta'_{n_k}\}\\
=&\min\{\liminf_{k\to+\infty}{\rm dens}_{U\cap\phi_{\alpha}(\mathbb{D}_{r_{\alpha}/A})}\Delta'_{n_k},
\liminf_{k\to+\infty}{\rm dens}_{U\cap\phi_{\alpha}(H(r_{\alpha}/A,r_{\alpha}))}\Delta'_{n_k}\}\\
\geq&\min\{1,1/2\}\\
=&1/2.
\end{flalign*}
At last, Theorem \ref{T1.1} follows from the arbitrariness of the subsequence $\{n_k\}_{k=1}^{\infty}$
with (\ref{F2.1}).

\section{The proof of Proposition \ref{P1}}
This section is devoted to proving Proposition \ref{P1}. 
The framework of the proof is similar to the original framework by Buff and Ch\'eritat, but the main different point is the use of near-parabolic renormalization technique instead of Yoccoz renormalization technique.
The next proof is divided into the following four subsections.

\subsection{\bf Perturbed Fatou coordinate}
Let
$$\epsilon_n:=\alpha_n-\frac{p_n}{q_n}=\frac{(-1)^n}{q_n^2(A_n+\theta)+q_nq_{n-1}}.$$%\stackrel{n\to\infty}{\sim}\frac{(-1)^n}{q_n^2A_n}.$$
Without loss of generality,
we assume that $A_n^{1/q_n}\to A\in(1,\infty]$ and $\epsilon_n>0$, say $n$ is even.
Let
$$l_n:=\{Z=X+Yi\in\mathbb{C}:X=\epsilon_n/2\}$$
and let $l_{n,j}$ ($j=1,2,\cdots,q_n$) denote $q_n$ connected components of $\sqrt[q_n]{l_n}$.
The complement $\mathbb{C}\setminus\cup_{j=1}^{q_n}l_{n,j}$ consists of $q_n+1$ simply connected components.
Let $s_0$ denote the component containing $0$ and $s_j$ ($j=1,2,\cdots,q_n$) denote another component with boundary $l_{n,j}$.
For all $j\in\{1,2,\cdots,q_n\}$, let $\zeta_j$ be the root of $\sqrt[q_n]{\epsilon_n}$ contained in $s_j$ and
$E_{n,j}:=(s_0\cup l_{n,j}\cup s_j)\setminus\{0,\zeta_j\}$.
Let
$$Q_n:=(-1/(q_n^2\epsilon_n),0)\cup\mathbb{H}^+\cup\mathbb{H}^-$$
with
$\mathbb{H}^+:=\{z\in\mathbb{C}:{\rm Im}(z)>0\}$ and $\mathbb{H}^-:=\{z\in\mathbb{C}:{\rm Im}(z)<0\}$.
It is not hard to see that
$$\pi_n(Z):=\sqrt[q_n]{\frac{\epsilon_n}{1-e^{-2\pi iq_n^2\epsilon_nZ}}}$$
is well defined on $Q_n$ and each $E_{n,j}$
is the image of $Q_n$ under the corresponding branch of $\pi_n$.
Given $j\in\{1,2,\cdots,q_n\}$,
we let $\pi_{n,j}$ denote the branch of $\pi_n$ whose image is $E_{n,j}$.
About $\pi_{n,j}$, the following properties are not hard to be checked:
\begin{itemize}
\item[$\bullet$] the restriction of $\pi_{n,j}$ to $\mathbb{H}^+$ is a covering from $\mathbb{H}^+$ to $s_0\setminus\{0\}$;
\item[$\bullet$] the restriction of $\pi_{n,j}$ to $\mathbb{H}^-$ is a covering from $\mathbb{H}^-$ to $s_j\setminus\{\zeta_j\}$;
\item[$\bullet$] the restriction of $\pi_{n,j}$ to $(-1/(q_n^2\epsilon_n),0)$ is a homeomorphism
from $(-1/(q_n^2\epsilon_n),0)$ to $l_{n,j}$.
\end{itemize}
For any positive real number $a$, set
$$Q_n(a):=
\left\{z\in\mathbb{C}:{\rm Re}(z-a+\frac{1}{q_n^2\epsilon_n})>-|{\rm Im}(z-a+\frac{1}{q_n^2\epsilon_n})|,\ {\rm Re}(z+a)<|{\rm Im}(z+a)|\right\}$$
and for any $Z\in\mathbb{C}$ and any positive real number $r$, set
$$\mathbb{B}(Z,r):=\{W\in\mathbb{C}:|W-Z|<r\}.$$
Given $1/A<r_0<r'_0<1$, we require $n$ large enough so that $f_n$ is well defined and univalent on $\mathbb{D}_{r'_0}$.
Next, $f_n$ is seen as a function defined on $\mathbb{D}_{r'_0}$.
About $f_n$, we have

\begin{lemma}
\label{L3.1}For all $1/A<r<r'<r_0$, large enough $n$ and each $j\in\{1,2,\cdots,q_n\}$,
we have that for all $Z\in Q_n(1/(2\pi q_n^2r^{q_n}))$,
\begin{itemize}
\item[{\rm(a)}] $\mathbb{B}(Z,1/(2\pi q_n^2r'^{q_n}))\subset Q_n$ and
$|\pi_{n,j}(Z)|^{q_n}\leq M\cdot r^{q_n}$ as $n\to\infty$, where $M$ is a positive real number, independent of $Z$;
%\item[{\rm(b)}] $f_n^{\comp q_n}(\pi_{n,j}(Z))\in\pi_{n,j}(\mathbb{B}(Z,1/(2\pi q_n^2r'^{q_n})))$;
\item[{\rm(b)}] there exists a complex number $F_n(Z)\in \mathbb{B}(Z,1/(2\pi q_n^2r'^{q_n}))$
such that
$$\pi_{n,j}(F_n(Z))=f_n^{\comp q_n}(\pi_{n,j}(Z))\ {\rm and}\
|F_n(Z)-Z-1|<\frac{1}{4}.$$
\end{itemize}
\end{lemma}
\begin{proof}
We assume that $n$ is large enough and independent of $Z\in Q_n(1/(2\pi q_n^2r^{q_n}))$.
By an immediate computation, we can obtain that
(a) holds and $\pi_{n,j}$ is univalent on $\mathbb{B}(Z,1/(2\pi q_n^2r'^{q_n}))$.

Define
$$g(W):=\frac{\pi_{n,j}(W)-\pi_{n,j}(Z)}{\pi'_{n,j}(Z)}+\pi_{n,j}(Z)$$
with
$W\in\mathbb{B}(Z,1/(2\pi q_n^2r'^{q_n}))$.
It is easy to check that
$g$ is a univalent map with $g(Z)=\pi_{n,j}(Z)$ and $g'(Z)=1$.
Thus it follows from Koebe $1/4$-distortion theorem that $g(\mathbb{B}(Z,1/(2\pi q_n^2r'^{q_n})))$
contains $\mathbb{B}(\pi_{n,j}(Z),1/(8\pi q_n^2r'^{q_n}))$.
By [Corollary $2$ and Lemma $3$, \cite{BC12}],
for large enough $n$, we have that
for all $z\in\mathbb{D}_{r_0}$ and $0\leq j\leq q_n$, $f_n^{\comp j}(z)\in\mathbb{D}_{r'_0}$
and there exist a complex number $\eta_n$ and a holomorphic function $k_n$,
defined on $\mathbb{D}_{r_0}$, such that for all $z\in\mathbb{D}_{r_0}$,
\[f_n^{\comp q_n}(z)=z+2\pi iq_nz(\epsilon_n-z^{q_n})(1+\eta_n+(\epsilon_n-z^{q_n})k_n(z)).\]
Moreover, there exists a sequence of positive real numbers $\{B_n\}_{n=1}^{+\infty}$ with
$\limsup_{n\to\infty}\sqrt[q_n]{B_n}\leq1$
such that
\[|\eta_n|\leq B_n\cdot|\epsilon_n|\ {\rm and}\ |k_n(z)|\leq B_n\ {\rm for\ all}\ z\in\mathbb{D}_{r_0}.\]
By (a), $\pi_{n,j}(Z)\in\mathbb{D}_{r_0}$ and hence
\begin{flalign*}
\frac{f_n^{\comp q_n}(\pi_{n,j}(Z))-\pi_{n,j}(Z)}{\pi'_{n,j}(Z)}+\pi_{n,j}(Z)=&\pi_{n,j}(Z)+1+\eta_n+(\epsilon_n-\pi_{n,j}(Z)^{q_n})k_n(\pi_{n,j}(Z))\\
=&\pi_{n,j}(Z)+1+o(1)\\
\in&\mathbb{B}(\pi_{n,j}(Z),1/(8\pi q_n^2r'^{q_n})).
\end{flalign*}
Then there exists $W\in\mathbb{B}(Z,1/(2\pi q_n^2r'^{q_n}))$ such that
$$g(W)=\frac{f_n^{\comp q_n}(\pi_{n,j}(Z))-\pi_{n,j}(Z)}{\pi'_{n,j}(Z)}+\pi_{n,j}(Z)$$
and
$|W-Z-1|=o(1)$ as $n\to\infty$, independent of $Z$.
This implies that
$f_n^{\comp q_n}(\pi_{n,j}(Z))=\pi_{n,j}(W)$ and
$|W-Z-1|<\frac{1}{4}$.
We take $F_n(Z)=W$, then
$\pi_{n,j}(F_n(Z))=f_n^{\comp q_n}(\pi_{n,j}(Z))$ and $|F_n(Z)-Z-1|<\frac{1}{4}$.
Thus (b) holds.
\end{proof}

Given $1/A<r_1<r<r'<r_0$, for large enough $n$, by Lemma \ref{L3.1}
for all $Z\in Q_n(1/(2\pi q_n^2r^{q_n}))$ and each $j\in\{1,2,\cdots,q_n\}$,
there exists a complex number $F_n(Z)\in \mathbb{B}(Z,1/(2\pi q_n^2r'^{q_n}))$
such that
$$\pi_{n,j}(F_n(Z))=f_n^{\comp q_n}(\pi_{n,j}(Z))\ {\rm and}\
|F_n(Z)-Z-1|<\frac{1}{4}.$$
From now on, instead of using $\pi_{n,j}$, we use $\pi_n$ to represent any $\pi_{n,j}$.
It is easy to see that if $\pi_{n}(Z)=\pi_{n}(Z')$ with $Z\not=Z'$,
then $|Z-Z'|\geq 1/(q_n^2\epsilon_n)$.
Then it follows from $|F_n(Z)-Z-1|<\frac{1}{4}$ that for large enough $n$,
$F_n(Z)$ is uniquely determined.
Thus for large enough $n$, $F_n:Q_n(1/(2\pi q_n^2r^{q_n}))\to\mathbb{C},\ Z\mapsto F_n(Z)$
is a function with
$$\pi_{n}(F_n(Z))=f_n^{\comp q_n}(\pi_{n}(Z))\ {\rm and}\
|F_n(Z)-Z-1|<\frac{1}{4}.$$
It follows from $\pi_{n}(F_n(Z))=f_n^{\comp q_n}(\pi_{n}(Z))$ that
$F_n$ is holomorphic.
Since $0<r_1<r$ and for large enough $n$,
$|F_n(Z)-Z-1|<\frac{1}{4}$ for all $Z\in Q_n(1/(2\pi q_n^2r^{q_n}))$,
we have that for large enough $n$,
$$|F'_n(Z)-1|<\frac{1}{4}\ {\rm for\ all}\ Z\in Q_n(1/(2\pi q_n^2r_1^{q_n})).$$
Thus $F_n$ is a holomorphic function on $Q_n(1/(2\pi q_n^2r_1^{q_n}))$
with
$$\pi_{n}(F_n(Z))=f_n^{\comp q_n}(\pi_{n}(Z)),\
|F_n(Z)-Z-1|<\frac{1}{4}\ {\rm and}\
|F'_n(Z)-1|<\frac{1}{4}.$$
By [Proposition A.2.1 and Lemma A.2.4 (i), \cite{Shi98}] we have the following (c), (d) and (e):
for large enough $n$,
\begin{itemize}
\item[(c)] $F_n$ is univalent;
\item[(d)] there exists a univalent function $\Phi_n$, defined on $Q_n(a_n)$, such that
$$\Phi_n\comp F_n(Z)=\Phi_n(Z)+1,\ Z\in Q_n(a_n)\cap F_n^{-1}(Q_n(a_n)),$$
$${\rm Im}(\Phi_n(Z))\to+\infty,\ {\rm as}\ Z\in Q_n(a_n)\ {\rm and}\ {\rm Im}(Z)\to+\infty$$
and
$${\rm Im}(\Phi_n(Z))\to-\infty,\ {\rm as}\ Z\in Q_n(a_n)\ {\rm and}\ {\rm Im}(Z)\to-\infty,$$
where $a_n=1/(2\pi q_n^2r_1^{q_n})$;
\item[(e)] let $Z_0\in Q_n(a_n)$ with ${\rm Re}(Z_0)=-\frac{1}{\pi q_n^2r_1^{q_n}}$. Denote
by $S$ the closed region (a strip) bounded by the two curves $l:=\{Z_0+iY: Y\in\mathbb{R}\}$
and $F_n(l)$. Then for any $Z\in Q_n(a_n)$, there exists a unique $m\in\mathbb{Z}$ such that $F_n^{\comp m}(Z)$ is
defined and belongs to $S\setminus F_n(l)$.
Moreover, $\exp\comp\Phi_n$ induces an isomorphism between $S/$$\scriptstyle(Z\sim F_n(Z))$
and $\mathbb{C}^*:=\mathbb{C}\setminus\{0\}$, where $\exp(Z):=e^{2\pi iZ}$.
\end{itemize}
The map $\Phi_n$ is called a {\bf perturbed Fatou coordinate} of $F_n$ and is determined uniquely by $f_n$
up to a translation. We set
$$\Phi_n(-\frac{1}{\pi q_n^2r_1^{q_n}})=
-\frac{1}{\pi q_n^2r_1^{q_n}}.$$

Next, we will extend $F_n$ and $\Phi_n$ to a bigger region than $Q_n(a_n)$.
Given $1/A<r_3<r'_2<r_2<r_1$, set
$$\mathbb{K}_n(r_2,r_3):=\left\{Z=X+Yi\in\mathbb{C}:
-\frac{1}{q_n^2\epsilon_n}-\frac{1}{2\pi q_n^2r_3^{q_n}}\leq X\leq\frac{1}{2\pi q_n^2r_3^{q_n}}, Y\geq\frac{1}{2\pi q_n^2r_2^{q_n}}\right\}$$
%$$\mathbb{K}_n(r'_2,r_3):=\left\{Z=X+Yi\in\mathbb{C}:
%-\frac{1}{q_n^2\epsilon_n}-\frac{1}{2\pi q_n^2r_3^{q_n}}\leq X\leq\frac{1}{2\pi q_n^2r_3^{q_n}}, Y\geq\frac{1}{2\pi q_n^2r_2'^{q_n}}\right\}$$
and
$$QB_n:=Q_n(a_n)\cup\mathbb{K}_n(r_2,r_3).$$
An immediate computation gives
$\pi_n(QB_n)\subset\mathbb{D}_{r_0}$ for large enough $n$.
The following result tells us that $F_n$ can be univalently extended to $QB_n$.
\begin{lemma}
\label{L3}
For large enough $n$,
$F_n$ can be univalently extended to $QB_n$
and for all $Z\in QB_n$, we have
$$\pi_n\comp F_n(Z)=f_n^{\comp q_n}\comp\pi_n(Z)\ {\rm and}\ |F_n(Z)-Z-1|<1/4.$$
Moreover, as long as $n$ is large enough, for all $Z\in\mathbb{K}_n(r_2',r_3)$ {\rm(}defined in the same way as $\mathbb{K}_n(r_2,r_3)${\rm)},
there exists an integer $m\geq0$ such that at least one of the following two conditions holds:
\begin{itemize}
\item[$\bullet$] $F_n^{\comp k}(Z)\in QB_n$ for $0\leq k\leq m$
and $F_n^{\comp m}(Z)\in Q_n(a_n)$;
\item[$\bullet$] $F_n^{\comp -k}(Z)\in QB_n$ for $0\leq k\leq m$
and $F_n^{\comp -m}(Z)\in Q_n(a_n)$.
\end{itemize}
\end{lemma}
\begin{proof}
Set
$$\mathbb{H}_n(r_1):=\{Z\in\mathbb{C}:{\rm Im}(Z)>\tau_n(r_1)\}$$
with
\[\tau_n(r_1):=\frac{1}{2\pi q_n^2\epsilon_n}\log(1+\frac{\epsilon_n}{r_1^{q_n}})\stackrel{n\to\infty}{\sim}\frac{1}{2\pi q_n^2r_1^{q_n}}.\]
It follows from a computation that for large enough $n$,
$\pi_n(\mathbb{H}_n(r_1))\subset\mathbb{D}_{r_0}$.
By [Propositions 8 and 9, \cite{BC12}], if $n$ is large enough,
there exists a holomorphic map $F_n:\mathbb{H}_n(r_1)\to\mathbb{H}^+$ such that
%\begin{itemize}
%\item
$$\pi_{n}\comp F_n(Z)=f_n^{\comp q_n}\comp\pi_{n}(Z)\ {\rm and}\
%\item
\sup_{Z\in\mathbb{H}_n(r_1)}\{|F_n(Z)-Z-1|\}\to0\ {\rm as}\ n\to\infty.$$
%\item $F_n(Z)=Z+1+o(1)$ as ${\rm Im}(Z)\to+\infty$.
%\end{itemize}
This implies that for large enough $n$,
$F_n$ can extend to $\mathbb{H}_n(r_1)\cup Q_n(a_n)$ with
\begin{equation}
\label{F3.1}\pi_{n}\comp F_n(Z)=f_n^{\comp q_n}\comp\pi_{n}(Z)\ {\rm and}\ |F_n(Z)-Z-1|<1/4.
\end{equation}
Next, we assume $n$ is large enough.
Since the restriction of $F_n$ to $\mathbb{H}_n(r_1)$ is a lift of $f_n^{\comp q_n}$ through $\pi_n$
and $f_n^{\comp q_n}$ fixing $0$ is inverse, we have that
$F_n$ is univalent on $\mathbb{H}_n(r_1)$.
If $Z,Z'\in\mathbb{H}_n(r_1)\cup Q_n(a_n)$ with $Z\not=Z'$
and $F_n(Z)=F_n(Z')$, then
$$f_n^{\comp q_n}\comp\pi_{n}(Z)=f_n^{\comp q_n}\comp\pi_{n}(Z')$$
and hence $Z-Z'=\frac{k}{q_n\epsilon_n}$ for some nonzero integer $k$.
Since $F_n$ is univalent on $Q_n(a_n)$, we have $Z,Z'\in\mathbb{H}_n(r_1)$
and this contradicts the fact that $F_n$ is univalent on $\mathbb{H}_n(r_1)$.
Thus $F_n$ is univalent on $\mathbb{H}_n(r_1)\cup Q_n(a_n)$.
Since $QB_n\subset\mathbb{H}_n(r_1)\cup Q_n(a_n)$, we have that
$F_n$ can be univalently extended to $QB_n$ with (\ref{F3.1}).

By [proofs of Lemmas 4 and 5, \cite{BC12}], we have that
for large enough $n$,
for any $Z\in\mathbb{K}_n(r_2',r_3)$ and $F_n^{\comp j}(Z)\in\mathbb{K}_n(r_2,r_3)$ for $j=1,2,\cdots,k$, we have
$$|{\rm Im}(F_n^{\comp k}(Z))-{\rm Im}(Z)|\leq2\int_{{\rm Re}(Z)}^{{\rm Re}(F_n^{\comp k}(Z))}u_n(X){\rm d}X$$
with
\[u_n(X):=\frac{7B_n\epsilon_n}{|(1+\frac{\epsilon_n}{r_2^{q_n}})e^{2\pi iq_n^2\epsilon_n X}-1|}\]
and $\{B_n\}_{n=1}^{\infty}$ is a sequence of positive real numbers with
$$\limsup_{n\to\infty}\sqrt[q_n]{B_n}\leq1.$$
If $${\rm Re}(F_n^{\comp k}(Z))-{\rm Re}(Z)<\frac{1}{\pi q_n^2r_3^{q_n}},$$
then we have
\begin{flalign*}
|{\rm Im}(F_n^{\comp k}(Z))-{\rm Im}(Z)|\leq&2\int_{{\rm Re}(Z)}^{{\rm Re}(F_n^{\comp k}(Z))}u_n(X){\rm d}X\\
\leq&2\int_{{\rm Re}(Z)}^{{\rm Re}(Z)+\frac{1}{\pi q_n^2r_3^{q_n}}}u_n(X){\rm d}X\\
=&14B_n\epsilon_n\int_{{\rm Re}(Z)}^{{\rm Re}(Z)+\frac{1}{\pi q_n^2r_3^{q_n}}}
\frac{1}{|(1+\frac{\epsilon_n}{r_2^{q_n}})e^{2\pi iq_n^2\epsilon_n X}-1|}{\rm d}X\\
=&\frac{7B_n}{\pi q_n^2}\int_{2\pi q_n^2\epsilon_n{\rm Re}(Z)}^{2\pi q_n^2\epsilon_n{\rm Re}(Z)+\frac{2\epsilon_n}{r_3^{q_n}}}
\frac{1}{|(1+\frac{\epsilon_n}{r_2^{q_n}})e^{i\theta}-1|}{\rm d}\theta\\
%\end{flalign*}
%\begin{flalign*}
\leq&\frac{7B_n}{\pi q_n^2}\int_{2\pi q_n^2\epsilon_n{\rm Re}(Z)}^{2\pi q_n^2\epsilon_n{\rm Re}(Z)+\frac{2\epsilon_n}{r_3^{q_n}}}
\frac{1}{\sqrt{(\frac{\epsilon_n}{r_2^{q_n}})^2+(2\sin\frac{\theta}{2})^2}}{\rm d}\theta\\
\leq&\frac{14B_n}{\pi q_n^2}\int_{0}^{\frac{2\epsilon_n}{r_3^{q_n}}}
\frac{1}{\sqrt{(\frac{\epsilon_n}{r_2^{q_n}})^2+(\frac{2\theta}{\pi})^2}}{\rm d}\theta\\
=&\frac{7B_n}{q_n^2}\int_{0}^{\frac{4r_2^{q_n}}{\pi r_3^{q_n}}}
\frac{1}{\sqrt{1+t^2}}{\rm d}t\\
\leq&\frac{7B_n}{q_n^2}(1+\log(\frac{4r_2^{q_n}}{\pi r_3^{q_n}})).
\end{flalign*}
Thus we have\\
(1) For any $Z\in\mathbb{K}_n(r_2',r_3)$ with ${\rm Re}(Z)<a_n-\frac{1}{q_n^2\epsilon_n}$,
if $F_n^{\comp j}(Z)\in\mathbb{K}_n(r_2,r_3)$ for $j=1,2,\cdots,k$ and
${\rm Re}(F_n^{\comp k}(Z))\leq a_n-\frac{1}{q_n^2\epsilon_n}$,
then
$${\rm Re}(F_n^{\comp k}(Z))-{\rm Re}(Z)<\frac{1}{\pi q_n^2r_3^{q_n}}$$
and hence
\begin{flalign*}
{\rm Im}(F_n^{\comp k}(Z))
&\geq{\rm Im}(Z)-\frac{7B_n}{q_n^2}(1+\log(\frac{4r_2^{q_n}}{\pi r_3^{q_n}}))\\
&\geq\frac{1}{2\pi q_n^2r_2'^{q_n}}-\frac{7B_n}{q_n^2}(1+\log(\frac{4r_2^{q_n}}{\pi r_3^{q_n}}))\\
&\stackrel{n\to\infty}{\sim}\frac{1}{2\pi q_n^2r_2'^{q_n}}.
\end{flalign*}
This implies $F_n^{\comp(k+1)}(Z)\in\mathbb{K}_n(r_2,r_3)$.
Thus there exists a positive integer $m$ such that
$F_n^{\comp j}(Z)\in QB_n$ for $j=1,2,\cdots,m$ and $F_n^{\comp m}(Z)\in Q_n(a_n)$.\\
(2) For any $Z\in\mathbb{K}_n(r_2',r_3)$ with $a_n-\frac{1}{q_n^2\epsilon_n}\leq{\rm Re}(Z)\leq-a_n$,
$m=0$ satisfies our requirement.\\
(3) For any $Z\in\mathbb{K}_n(r_2',r_3)$ with $-a_n<{\rm Re}(Z)$,
if $F_n^{\comp -j}(Z)\in\mathbb{K}_n(r_2,r_3)$ for $j=1,2,\cdots,k$ and
${\rm Re}(F_n^{\comp -k}(Z))\geq-a_n$,
then
$${\rm Re}(Z)-{\rm Re}(F_n^{\comp -k}(Z))<\frac{1}{\pi q_n^2r_3^{q_n}}$$
and by the similar estimation used in the case (1) we can deduce
$F_n^{\comp(-k-1)}(Z)\in\mathbb{K}_n(r_2,r_3)$.
%hence
%$${\rm Im}(F_n^{\comp-k}(Z))\geq{\rm Im}(Z)-\frac{7B_n}{q_n^2}(1+\log(\frac{4r_2^{q_n}}{\pi r_3^{q_n}}))
%\geq\frac{1}{2\pi q_n^2r_2'^{q_n}}-\frac{7B_n}{q_n^2}(1+\log(\frac{4r_2^{q_n}}{\pi r_3^{q_n}}))
%\stackrel{n\to\infty}{\sim}\frac{1}{2\pi q_n^2r_2'^{q_n}}.$$
%\begin{flalign*}
%{\rm Im}(F_n^{\comp-k}(Z))
%&\geq{\rm Im}(Z)-\frac{7B_n}{q_n^2}(1+\log(\frac{4r_2^{q_n}}{\pi r_3^{q_n}}))\\
%&\geq\frac{1}{2\pi q_n^2r_2'^{q_n}}-\frac{7B_n}{q_n^2}(1+\log(\frac{4r_2^{q_n}}{\pi r_3^{q_n}}))\\
%&\stackrel{n\to\infty}{\sim}\frac{1}{2\pi q_n^2r_2'^{q_n}}.
%\end{flalign*}
%This implies $F_n^{\comp(-k-1)}(Z)\in\mathbb{K}_n(r_2,r_3)$.
Thus there exists a positive integer $m$ such that
$F_n^{\comp-j}(Z)\in QB_n$ for $j=1,2,\cdots,m$ and $F_n^{\comp-m}(Z)\in Q_n(a_n)$.
\end{proof}
By Lemma \ref{L3} and (e), we have
\begin{corollary}
\label{C3.1}For large enough $n$,
$\Phi_n$ can be univalently extended to $QB'_n:=\mathbb{K}_n(r'_2,r_3)\cup Q_n(a_n)$
and for all $Z\in QB'_n\cap F_n^{-1}(QB'_n)$,
$\Phi_n\comp F_n(Z)=\Phi_n(Z)+1.$
Moreover, if $Z,Z'\in QB'_n$ with $\Phi_n(Z')=\Phi_n(Z)+m$, then $Z'=F_n^{\comp m}(Z)$.
\end{corollary}

\subsection{\bf The renormalization associated to perturbed Fatou coordinate}
Set $Z_n:=X_n+iY_n$,
where $-1/(2\pi q_n^2r_5^{q_n})<X_n<1/(2\pi q_n^2r_5^{q_n})$ and $Y_n=1/(2\pi q_n^2r_4^{q_n})$
with $r_3<r_5<r_4<r'_2$.
Let
$$\mathcal{L}(Z_n)=\{Z=X+Yi\in\mathbb{C}:\ X=X_n,\ Y\geq Y_n\}$$
and
$$[Z_n,F_n(Z_n)]=\{tZ_n+(1-t)F_n(Z_n): 0\leq t\leq 1\}.$$
For large enough $n$,
the union
$$\mathcal{L}(Z_n)\cup[Z_n,F_n(Z_n)]\cup F_n(\mathcal{L}(Z_n))\cup\{\infty\}$$
forms a Jordan curve on the Riemann sphere. By $\mathcal{H}_n(Z_n)$ we denote
the closure (in $\mathbb{C}$) of the region bounded by this Jordan curve and not containing $0$.
%satisfies $\Im(\mathcal{H}_n(Z_0))>-\infty$.

To define the renormalization, we need to lift $f_n^{\comp q_{n-1}}$ to the $Z$-coordinate through $\pi_n$.
Recall
$$\mathbb{H}_n(r_1)=\{Z\in\mathbb{C}:{\rm Im}(Z)>\tau_n(r_1)\}$$
with
\[\tau_n(r_1)=\frac{1}{2\pi q_n^2\epsilon_n}\log(1+\frac{\epsilon_n}{r_1^{q_n}})\stackrel{n\to\infty}{\sim}\frac{1}{2\pi q_n^2r_1^{q_n}}.\]
%Let
%$$\mathbb{H}_n(r_4)=\{Z=X+Yi\in\mathbb{C}:\ Y\geq1/(2\pi q_n^2r_4^{q_n})\}.$$
It follows from [Proposition 8, \cite{BC12}] that
%\begin{lemma}[BC]
there exists a holomorphic map
$G_n:\mathbb{H}_n(r_1)\to\mathbb{H}^+$ if $n$ is large enough, such that
\begin{itemize}
\item[$\bullet$] The map $G_n$ is semi-conjugated to $f_n^{\comp q_{n-1}}$ by $\pi_n$:
\[\pi_n\comp G_n=f_n^{\comp q_{n-1}}\comp\pi_n.\]
%\item $G_n-Id$ is periodic of period $1/(q_n\epsilon_n)$.
\item[$\bullet$] As ${\rm Im}(Z)\to+\infty$, we have
\[G_n(Z)=Z-(A_n+\theta)+o(1).\]
\item[$\bullet$] The sequence
\[\sup_{Z\in\mathbb{H}_n(r_1)}\left|G_n(Z)-Z+A_n+\theta\right|\]
are sub-exponential with respect to $q_n$, that is,
\[\limsup_{n\to\infty}\sqrt[q_n]{\sup_{Z\in\mathbb{H}_n(r_1)}\left|G_n(Z)-Z+A_n+\theta\right|}\leq1.\]
\end{itemize}
%For all $Z\in\mathcal{H}_n(Z_n)$, $G_n(Z)\in QB_n$.
It is not hard to check that for large enough $n$,
$F_n\comp G_n(Z)=G_n\comp F_n(Z)$ for all $Z\in\mathbb{H}_n(r_2)$.
Next, we assume that $n$ is large enough.
Since
$$A_n+\theta=\frac{1}{q_n^2\epsilon_n}-\frac{q_{n-1}}{q_n},$$
we have that for all $Z\in\mathcal{H}_n(Z_n)$, $G_n(Z)\in QB'_n$. Then
the {\bf renormalized map} is defined by
$$\mathcal{R}(f_n)(z)=\exp\comp\Phi_n\comp G_n\comp\Phi_n^{-1}\comp\exp^{-1}(z),\ z\in\exp(\Phi_n(\mathcal{H}_n(Z_n))),$$
where $\exp^{-1}(z)\in\Phi_n(\mathcal{H}_n(Z_n))$.
Combining $F_n\comp G_n(Z)=G_n\comp F_n(Z)$ with Corollary \ref{C3.1},
one can check that the renormalized map is well defined and holomorphic on
$\exp(\Phi_n(\mathcal{H}_n(Z_n)))$, isomorphic to a punctured disk, and
the origin is a removable singularity of $\mathcal{R}(f_n)$.
Since $G_n$ is a lift of $f_n^{\comp q_{n-1}}$ through $\pi_n$
and $f_n^{\comp q_{n-1}}$ fixing $0$ is inverse, we have that
$G_n$ is univalent on $\mathbb{H}_n(r_1)$.
Then it is not hard to check that $\mathcal{R}(f_n)$ is
a univalent map fixing the origin defined on $\exp(\Phi_n(\mathcal{H}_n(Z_n)))\cup\{0\}$.
We observe that for large enough $n$,
$F_n(Z)=Z+1+o(1)$ as $Z\in QB_n$ and ${\rm Im}(Z)\to+\infty$.
Then by [Lemma A.2.4 (i), \cite{Shi98}] we have
$|\Phi'_n(Z)-1|\to0$ as $Z\in\mathbb{K}_n(r'_2,r_3)$ and ${\rm Im}(Z)\to+\infty$.
Thus
$$\Phi_n(G_n(\Phi^{-1}_n(Z)))-Z=\int_{\Phi^{-1}_n(Z)}^{G_n(\Phi^{-1}_n(Z))}\Phi'_n(z){\rm d}z\to-(A_n+\theta)$$
as $Z\in\Phi_n(\mathcal{H}_n(Z_n))$ and ${\rm Im}(Z)\to+\infty$.
This implies $\mathcal{R}(f_n)'(0)=e^{-2\pi i\theta}$.
Let $\rho_n$ be a positive real number such that
%for large enough $n$,
$\mathbb{D}_{\rho_n}$ is the largest disk contained completely in $\exp(\Phi_n(\mathcal{H}_n(Z_n)))\cup\{0\}$.
%It follows from the compactness that there exists a universal constant $c>0$ such that
%$\rho\geq e^{-c/(q_n^2r_4^{q_n})}$.
According to the theorem on page $21$ in \cite{Yoc95}, there exists a constant $c>0$ depending only on $\theta$ such that
$\mathbb{D}_{c\rho_n}$ is contained in
the Siegel disk of $\mathcal{R}(f_n)$ centering at the origin.

\subsection{\bf Iterating the commuting pair $(F_n,G_n)$ through the renormalization\label{s3.3}}
Set
$$\mathbb{H}_n(r_6,r_5):=\{Z=X+Yi\in\mathbb{C}:|X|<\frac{1}{2\pi q_n^2r_5^{q_n}},\ Y>\frac{1}{2\pi q_n^2r_6^{q_n}}\}$$
with $r_5<r_6<r_4$. Assume that $n$ is large enough.
For all $Z=X+iY\in\mathbb{H}_n(r_6,r_5)$,
we have
$Z\in\mathcal{H}_n(Z_n)$,
where $Z_n=X_n+iY_n$ with
$X_n=X$ and $Y_n=1/(2\pi q_n^2r_4^{q_n})$.
Set $z:=\exp(\Phi_n(Z))$. Similar to $(b)$ of Step $3$ in the proof of [Proposition 10, \cite{BC12}],
one can prove that there exists an annulus $\mathcal{A}_n(Z)\subset\mathbb{C}$ such that
\begin{itemize}
\item[$\bullet$] the annulus $\mathcal{A}_n(Z)$ separates $0$ and $z$ from $\infty$ and a point of modulus $\rho_n$ in
$\partial\exp\comp\Phi_n(\mathcal{H}_n(Z_n))$;
\item[$\bullet$] the conformal modulus of $\mathcal{A}_n(Z)$ converges uniformly to $+\infty$ on
$\mathbb{H}_n(r_6,r_5)$ as $n\to\infty$.
\end{itemize}
This implies that $|\frac{z}{\rho_n}|$ converges uniformly to $0$ on $\mathbb{H}_n(r_6,r_5)$ as $n\to\infty$.
%In fact,
%we set $Z'_n:=X+i/(2\pi q_n^2r_6^{q_n})$ and denote by $L$ the Jordan curve
%$[Z_n,F_n(Z_n)]\cup[Z_n,Z'_n]\cup[Z'_n,F_n(Z'_n)]\cup F_n([Z_n,Z'_n])$.
%Let $D_n$ be the union of the bounded Jordan domain bounded by $L$ and the `left boundary' $[Z_n,Z'_n]$.
%Then the image $\mathbb{R}_n=\exp(\Phi_n(D_n))$ is an annulus separating $0$ and $z$
%from $\infty$ and $\exp(\Phi_n([Z_n,F_n(Z_n)]))$.
%We claim that the module of $\mathbb{R}_n$ converges to $+\infty$ as $n\to\infty$.
%In fact, Then $|z|=o(\rho_n)$ and
Thus $z$ is contained in $\mathbb{D}_{c\rho_n}$.
Then $z$ is contained in
the Siegel disk of $\mathcal{R}(f_n)$ centering at the origin
and hence
$$\{\mathcal{R}(f_n)^{\comp m}(z)\}_{m=0}^{+\infty}\subset\exp(\Phi_n(\mathcal{H}_n(Z_n))).$$
Then it follows from Corollary \ref{C3.1} that
there exists a sequence $\{n_j\}_{j=1}^{+\infty}$ of positive integers such that
for all $t\geq0$ and $0\leq s<n_{t+1}$, we have
$$F_n^{\comp s}\comp G_n\comp F_n^{\comp n_t}\comp\cdots\comp G_n\comp F_n^{\comp n_1}\comp G_n(Z)\in QB_n.$$
Thus for all $t\geq0$ and $0\leq s<n_{t+1}$,
$$f_n^{\comp(sq_n)}\comp f_n^{\comp q_{n-1}}\comp f_n^{\comp(n_tq_n)}\comp\cdots\comp
f_n^{\comp q_{n-1}}\comp f_n^{\comp(n_1q_n)}\comp f_n^{\comp q_{n-1}}(\pi_n(Z))\in\pi_n(QB_n)\subset\mathbb{D}_{r_0}.$$
%It is easy to check that there exists a common $N$ such that for all $n\geq N$ and
This implies that for all $l\geq0$,
$f_n^{\comp l}(\pi_n(Z))\in\mathbb{D}_{r'_0}$ and
hence the set $\pi_n(\mathbb{H}_n(r_6,r_5))$
is contained in the Siegel disk $\Delta_n(r'_0)$ of $f_n$ centering at the origin.

\subsection{\bf liminf of the area rate of perturbed Siegel disks}
We give the following two notations:
\begin{itemize}
%\item $\mathcal{S}_n(r_6,r_5):=\left\{Z\in\mathbb{C}:|\Re(Z)|<
%\frac{1}{2\pi q_n^2r_5^{q_n}},\ \frac{1}{2\pi q_n^2r_6^{q_n}}<\Im(Z)<\frac{1}{2\pi q_n^2r_5^{q_n}}\right\};$
\item[$\bullet$] $\mathcal{X}_n(r_8):=\{z\in\mathbb{C}:\frac{z^{q_n}}{z^{q_n}-\epsilon_n}\in\mathbb{D}_{s_n}\}$
with $r_5<r_8<r_6$ and $s_n=\frac{r_8^{q_n}}{r_8^{q_n}+\epsilon_n}$;
\item[$\bullet$] $\mathcal{Y}_n(r_8,r_7):=\mathcal{X}_n(r_8)\setminus\mathbb{D}_{r_7}$ with $r_5<r_7<r_8$.
\end{itemize}
About those notations, we have the following properties:
\begin{itemize}
\item[(1)] for any nonempty open subset $U\subset\mathbb{H}(r_7,r_8)$,
$\liminf_{n\to+\infty}{\rm dens}_{U}\mathcal{Y}_n(r_8,r_7)\geq\frac{1}{2}$;
\item[(2)] for large enough $n$, $\mathcal{Y}_n(r_8,r_7)$ is contained in
the union of $q_n$ branches of $\pi_n(\mathbb{H}_n(r_6,r_5))$.
\end{itemize}
\begin{proof}%[Proof of properties (1) and (2)]
(1): For all $r_7<t<r_8$ and $\theta_1<\theta_2$, set
$$l(t,\theta_1,\theta_2):=\{z\in\mathbb{C}:|z|=t\ {\rm and}\ \theta_1<\arg(z)<\theta_2\}.$$
Let $T_{n,1}(z):=z^{q_n}$ and $T_{n,2}(z):=\frac{z}{z-\epsilon_n}$. Then
$$T_{n,2}^{-1}(\mathbb{D}_{s_n})=
\{z\in\mathbb{C}:|z+\frac{r_8^{2q_n}}{2r_8^{q_n}+\epsilon_n}|<
\frac{r_8^{q_n}(r_8^{q_n}+\epsilon_n)}{2r_8^{q_n}+\epsilon_n}\}$$
and hence $\mathbb{B}(-\frac{r_8^{q_n}}{2},\frac{r_8^{q_n}}{2})\subset T_{n,2}^{-1}(\mathbb{D}_{s_n})$.
Since for all $\theta$ and all $r_7<t<r_8$,
$T_{n,1}(l(t,\theta,\theta+\frac{2\pi}{q_n}))=\partial\mathbb{D}_{t^{q_n}}$ and
the length of $\partial\mathbb{D}_{t^{q_n}}\cap\mathbb{B}(-\frac{r_8^{q_n}}{2},\frac{r_8^{q_n}}{2})$
is $t^{q_n}(\pi-\frac{2t^{q_n}}{r_8^{q_n}}+o(\frac{t^{q_n}}{r_8^{q_n}}))$ as $n\to\infty$,
we have that the length of $l(t,\theta,\theta+\frac{2\pi}{q_n})\cap\mathcal{Y}_n(r_8,r_7)$
is at least
$$\frac{2\pi}{q_n}\frac{\pi-\frac{2t^{q_n}}{r_8^{q_n}}+o(\frac{t^{q_n}}{r_8^{q_n}})}{2\pi}t\ {\rm as}\ n\to\infty.$$
Thus for all $\theta_1<\theta_2$ with $\theta_2-\theta_1\leq2\pi$ and all $r_7<t<r_8$,
the length of $l(t,\theta_1,\theta_2)\cap\mathcal{Y}_n(r_8,r_7)$ is at least
$$(\theta_2-\theta_1)t/2+o(1)+O(\frac{t^{q_n}}{r_8^{q_n}})\ {\rm as}\ n\to\infty.$$
%$(\theta_2-\theta_1)\frac{\pi-\frac{2t^{q_n}}{r_8^{q_n}}+o(\frac{t^{q_n}}{r_8^{q_n}})}{2\pi}t$.
This implies that
for any nonempty open subset $U$ compactly contained in $\mathbb{H}(r_7,r_8)$,
the area of $U\cap\mathcal{Y}_n(r_8,r_7)$ is at least ${\rm area}(U)/2+o(1)$ as $n\to\infty$.
Observe that any nonempty open subset of $\mathbb{H}(r_7,r_8)$ could be approximated by
nonempty open sets compactly contained in $\mathbb{H}(r_7,r_8)$.
Thus (1) holds.

(2): Assume that $n$ is large enough.
For all $z\in\mathcal{Y}_n(r_8,r_7)$, set $w:=\frac{z^{q_n}}{z^{q_n}-\epsilon_n}$ with $w=x+yi$.
Since $z\in\mathcal{X}_n(r_8)$, we have $w\in\mathbb{D}_{s_n}$, that is
\begin{equation}
\label{F4.1}x^2+y^2\leq s_n^2.
\end{equation}
Since $z\not\in\mathbb{D}_{r_7}$, an elementary computation gives
\begin{equation}
\label{F4.2}(x-\frac{r_7^{2q_n}}{r_7^{2q_n}-\epsilon_n^2})^2+y^2\leq(\frac{\epsilon_nr_7^{q_n}}{r_7^{2q_n}-\epsilon_n^2})^2.
\end{equation}
%Thus $w$ is determined completely by equations (\ref{F4.1}) and (\ref{F4.2}).
Combining (\ref{F4.1}) with (\ref{F4.2}), one can obtain
$$\frac{r_7^{q_n}}{r_7^{q_n}+\epsilon_n}<|w|<\frac{r_8^{q_n}}{r_8^{q_n}+\epsilon_n}$$
and
$$|\arg(w)|\leq\frac{\epsilon_n}{r_7^{q_n}}+{\rm o}(\frac{\epsilon_n}{r_7^{q_n}})\
{\rm as}\ n\to+\infty.$$
Since $n$ is large enough, we have that
there exists $W\in\mathbb{H}_n(r_6,r_5)$ such that $w=e^{2\pi iq_n^2\epsilon_nW}$,
that is $z^{q_n}=(\pi_n(W))^{q_n}$.
This implies that $z$ is contained in some branch of $\pi_n(\mathbb{H}_n(r_6,r_5))$.
Thus (2) holds.
\end{proof}
It follows from (2) and Subsection \ref{s3.3} that
for large enough $n$, $\mathcal{Y}_n(r_8,r_7)$ is contained in $\Delta_n(r'_0)$.
Then by (1) we have that for any nonempty open subset $U\subset\mathbb{H}(r_7,r_8)$,
\begin{flalign*}
\liminf_{n\to+\infty}{\rm dens}_U\Delta_n(r'_0)
\geq\liminf_{n\to+\infty}{\rm dens}_U\mathcal{Y}_n(r_8,r_7)
\geq\frac{1}{2}.
\end{flalign*}
Thus by the arbitrariness of $r_j$ ($j=0,1,\cdots,8$), $r$, $r'$, $r'_2$ and $r'_0$, we have that
for any nonempty open set $U\subset H(1/A,r_8)$ with any $1/A<r_8<r'_0<1$,
\begin{equation*}
\liminf_{n\to+\infty}{\rm dens}_U\Delta_n(r'_0)\geq\frac{1}{2}.
\end{equation*}
The proof of Proposition \ref{P1} is complete.

\end{document}